\documentclass[12pt]{article}
\usepackage{amsmath,amssymb,amstext,dsfont,fancyvrb,float,fontenc,graphicx,subfigure,theorem,hyperref,esint}
\usepackage[utf8]{inputenc}

 \usepackage{tikz,everypage}

\usepackage[letterpaper]{geometry}
\ifx\volno\undefined\def\volno{0}\fi
\ifx\volyear\undefined\def\volyear{2017}\fi
\ifx\pagno\undefined\def\pagno{000--000}\fi

\newfont{\footsc}{cmcsc10 at 8truept}
\newfont{\footbf}{cmbx10 at 8truept}
\newfont{\footrm}{cmr10 at 10truept}

\usepackage{fancyhdr}
\usepackage{relsize}
\usepackage{sectsty}
\allsectionsfont{\larger[-1]} 

\renewcommand\paragraph{\@startsection{paragraph}{4}{\z@}
                                    {2ex \@plus.5ex \@minus.2ex}
                                    {-1em}
                                    {\normalfont\normalsize\bfseries}}

\renewcommand\subparagraph{\@startsection{subparagraph}{5}{\parindent}
                                       {2ex \@plus.5ex \@minus .2ex}
                                       {-1em}
                                      {\normalfont\normalsize\bfseries}}

\newlength{\BiblioSpacing}
\renewenvironment{thebibliography}[1]{
\begin{oldthebibliography}{#1}
\setlength{\parskip}{\BiblioSpacing}
\setlength{\itemsep}{\BiblioSpacing}
}
{
\end{oldthebibliography}
}

\usepackage[strict]{changepage}
\def\abstractname{Abstract -}   
\def\abstract{\begin{adjustwidth}{1cm}{1cm} \par    \footnotesize \noindent {\bf \abstractname} 
\def\endabstract{ \end{adjustwidth} \smallskip }}

{\theorembodyfont{\itshape}\newtheorem{theorem}{Theorem}[section]}
{\theorembodyfont{\itshape}\newtheorem{proposition}[theorem]{Proposition}}
{\theorembodyfont{\itshape}}
{\theorembodyfont{\itshape}\newtheorem{lemma}[theorem]{Lemma}}
{\theorembodyfont{\itshape}\newtheorem{corollary}[theorem]{Corollary}}
{\theorembodyfont{\rm}}
{\theorembodyfont{\rm}\newtheorem{remark}[theorem]{Remark}}
{\theorembodyfont{\rm}}
{\theorembodyfont{\rm }}

\def \R {\mathbb{R}}
\def \C {\mathcal{C}}

\DeclareMathOperator{\Ric}{Ric}
\DeclareMathOperator{\Sec}{Sec}

\DeclareMathOperator{\diam}{diam}
\DeclareMathOperator{\vol}{vol}

\DeclareMathOperator{\diver}{div}

\newcommand{\norm}[1]{||#1||}
\newcommand{\ceil}[1]{\left\lceil #1 \right\rceil}
\newcommand{\N}{\mathbb{N}}
\newtheorem{claim}{Claim}

\newcommand{\drawingEps}{0.5}
\newcommand{\drawingR}{2}
\newcommand{\drawingXOne}{1.5}
\newcommand{\drawingXTwo}{{\drawingXOne + sqrt(\drawingR^2 - \drawingEps^2)}}
\newcommand{\drawingXTwoNeg}{{-\drawingXOne - sqrt(\drawingR^2 - \drawingEps^2)}}

\title{\Large\bf Manifolds with bounded integral curvature and no positive eigenvalue lower bounds}
\author{C. Anderson, X. Ramos Oliv\'e, and K. Spinelli}

\begin{document}
\setcounter{page}{1}
\maketitle
\thispagestyle{fancy}

\vskip 1.5em

\begin{abstract}
We provide an explicit construction of a sequence of closed surfaces with uniform bounds on the diameter and on
$L^p$ norms of the curvature, but without a positive lower bound on the first non-zero eigenvalue of the Laplacian
$\lambda_1$. This example shows that the assumption of smallness of the $L^p$ norm of the curvature is a necessary
condition to derive Lichnerowicz and Zhong-Yang type estimates under integral curvature conditions.
\end{abstract}
 

\begin{MSC}
58J50; 53C21; 58J60
\end{MSC}

\section{Introduction}

On an $n$-dimensional Riemannian manifold $(M^n,g)$, consider the Laplace-Beltrami operator $\Delta u := \diver(\nabla u)$ for $u\in \C^2(M)$. We say that $u \not \equiv 0$ is an \emph{eigenfunction of the Laplacian} on $M$ with \emph{eigenvalue} $\lambda \in \R$, if
\begin{equation}\label{eigenprob}
    -\Delta u =\lambda u\ \ \text{in }M.
\end{equation}
When $M$ is closed (compact without boundary), then the spectrum of $-\Delta$ is discrete and the eigenvalues can be arranged as a non-decreasing diverging sequence of non-negative numbers
\[0=\lambda_0 < \lambda_1 \leq \lambda_2 \leq \lambda_3 \leq \ldots \rightarrow \infty.\]
Note that the smallest eigenvalue is always $\lambda_0 =0$, since non-zero constant functions are trivial examples of eigenfunctions on $M$. It is interesting to estimate from below the first non-zero eigenvalue $\lambda_1$ in geometric terms: lower bounds on $\lambda_1$ are connected to Poincar\'e and Sobolev-Poincar\'e inequalities, and from a physical point of view, they correspond to estimating the fundamental energy gap.

In \cite{Lichnerowicz}, Lichnerowicz showed that if  $\Ric \geq (n-1)K$ for some $K>0$, where $\Ric$ denotes the Ricci curvature, then 
\begin{equation}\label{Lichnerowicz}
    \lambda_1 \geq nK.
\end{equation}
Notice that this is a uniform estimate on $\lambda_1$ independent of $(M,g)$, in the class of $n-$manifolds satisfying the curvature assumption. 

In the case when $K=0$, a similar estimate can be found when the diameter of $M$ (the largest distance between two points) is bounded above by $D> 0$. This is not an additional assumption, since when $\Ric \geq (n-1)K>0$, the diameter satisfies $\diam(M) \leq \frac{\pi}{\sqrt{K}}$ by Myers' Theorem. With the work of Li-Yau \cite{LiYau} and Zhong-Yang \cite{ZhongYang}, it was shown that if $\Ric \geq 0$ and $\diam(M) \leq D$, then
\begin{equation}\label{ZhongYang}
    \lambda_1 \geq \frac{\pi^2}{D^2}.
\end{equation}
Estimates \eqref{Lichnerowicz} and \eqref{ZhongYang} are sharp, since equality is achieved for the sphere $S^n$ and the circle $S^1$, respectively. In fact, \cite{Obata} and \cite{HangWang} proved rigidity results, i.e. they showed that $S^n$ and $S^1$, respectively, are the only manifolds in these classes where equality is achieved.

Lower bounds in the negative curvature case were studied by Yang \cite{Yang}. There is a very extensive literature of related results and alternative proofs by several authors (see for instance \cite{AndrewsClutterbuck}, \cite{ChenWang},  \cite{Kroger}, \cite{LingLu}, \cite{ShiZhang}, \cite{ZhangWang} and the references therein). 

Recently, there has been an increasing interest in relaxing the curvature assumption from a pointwise condition to an integral condition. Integral assumptions are more stable under perturbations of the metric, and are much weaker than pointwise assumptions. This is particularly interesting when considering sequences of manifolds converging in some sense to a limit space, like in the Gromov-Hausdorff sense (see \cite{BuragoBuragoIvanov} for a general introduction to the subject), or in the Sormani-Wenger intrinsic flat convergence sense (see \cite{SormaniWenger}). If the curvature assumption is more stable under perturbations of the metric, a sequence of slightly perturbed manifolds will satisfy the same curvature assumption uniformly, and then results like the ones above can be used to derive information about the limit space. An example of this is the  Cheeger-Colding spectral convergence theory \cite{CheegerColding}, for uniform lower bounds on $\rm Ric$. 

In \cite{Gallot} and \cite{Petersen-Sprouse1998}, estimates on the isoperimetric constant were obtained in terms of integral conditions on the Ricci curvature, which imply estimates on the Sobolev constant and on $\lambda_1$. Petersen-Wei \cite{PetersenWei} generalized the classical Bishop-Gromov volume comparison estimate to integral curvature, opening up the possibility of extending some of the classical results for pointwise bounds on the Ricci curvature to integral assumptions. 

To be more precise, let $\rho\left( x\right) $ be the smallest eigenvalue for the Ricci tensor at $x\in M$. For a constant $K \in \mathbb R$, let $\rho_K$ be the amount of Ricci curvature lying below $(n-1)K$, i.e.
\begin{equation}\label{rhoK}
\rho_K=\max\{-\rho(x)+(n-1)K, 0\}.
\end{equation}
In particular, $\rho_0 = \rho^-$ is the negative part of $\rho$. The following quantity measures the amount of Ricci curvature lying below $(n-1)K$ in an $L^p$ average sense:
\begin{equation}\label{integralcurv}
\bar k(p,K)=\left(\frac{1}{\vol(M)}\int_M \rho_K^p dv\right)^{\frac{1}{p}}= \left(\fint_M \rho_K^p dv\right)^{\frac{1}{p}}.    
\end{equation}
Notice that $\bar k(p,K) = 0$ if and only if $\Ric \ge (n-1)K$.
 
The first estimates on $\lambda_1$ in terms of $\bar k(p,K)$ are due to Gallot in the pioneering work \cite{Gallot}, although they are not sharp. A generalization of \eqref{Lichnerowicz} was proven by Aubry in \cite{Aubry}. For $K>0$ and $p>n/2$ one has
\begin{equation}\label{Aubry}
    \lambda_1 \geq nK\left(1-C(n,p)\bar k(p,K)\right).
\end{equation}

More recently, in \cite{RSWZ}, the second author together with Seto, Wei and Zhang generalized \eqref{ZhongYang}. If $\diam(M)<D$, for any $\alpha\in (0,1)$ and $p>n/2$ there exists $\epsilon(n,p,\alpha,D)>0$ such that if $\bar k(p,0)<\epsilon$ then
\begin{equation}\label{RSWZ}
    \lambda_1\geq \alpha \frac{\pi^2}{D^2}.
\end{equation}
Estimates \eqref{Aubry} and \eqref{RSWZ} recover \eqref{Lichnerowicz} and \eqref{ZhongYang} in the limit where $\Ric \geq (n-1)K>0$ and $\Ric\geq 0$, respectively, so they are in that sense sharp. Further extensions of \eqref{RSWZ} to a Kato-type condition on the Ricci curvature (an even weaker curvature assumption) recently appeared in the work of Rose-Wei \cite{RoseWei}. Another direction of generalization could be to follow the work of Wu \cite{Wu}, who studied sm\-ooth metric measure spaces under integral conditions on the Bakry-\'Emery Ricci curvature; eigenvalue estimates in this setting were proven in \cite{Ramosthesis}.

The motivation of the present article is to explore the necessity of smallness for $\bar k(p,K)$ in the above theorems. The necessity of the condition $p>n/2$ to derive an eigenvalue estimate was discussed by \cite{Gallot}. Several authors (see \cite{Gallot} and \cite{DaiWeiZhang}) have discussed the necessity of the smallness of $\bar k(p,K)$ in order to derive a bound of the first Betti number and to extend Cheeger-Colding theory to integral curvature conditions. Here we focus on the necessity of $\bar k(p,K)$ being small in order to obtain a positive lower bound on $\lambda_1$. 

The necessity of smallness of $\bar k(p,0)$ to derive lower bounds on any eigenvalue $\lambda_i$ had already been discussed in \cite{Gallot} (see Appendix A.5). Our approach here is more elementary, and we discuss why the necessity of $\bar k(p,0)$ being small implies that $\bar k(p,K)$ also needs to be small. Moreover, with our approach we see that, even under a stronger assumption on the curvature, like integral sectional curvature bounds, smallness would be necessary; this had not been explored previously in the literature to the best of our knowledge. 

The general idea that we will develop was outlined in \cite{RSWZ}, where the authors considered a sequence of manifolds along which $\lambda_1$ goes to $0$, and where $\bar k(p,0)$ can not be made small. However, in that paper the example discussed is not explicit, so it is not clear if the construction can be done in such a way that $\bar k(p,0)$ remains bounded along the sequence. We will construct an explicit sequence of manifolds along which $\bar k(p,0)$ is uniformly bounded, but not small, and for which $\lambda_1$ goes to $0$, showing that smallness of $\bar k(p,0)$ is necessary to generalize \eqref{ZhongYang}. Furthermore, we also show that $\bar k(p,K)$ is also uniformly bounded for any $K>0$, showing that the smallness of that quantity is also necessary to generalize \eqref{Lichnerowicz}. Alternative explicit constructions can also be found in \cite{Leal}, although their example had a different purpose and no estimates on $\bar k(p,K)$ were provided.

More precisely, our main theorem is

\begin{theorem}
    \label{mainthm}
 For any $K\geq 0$, there exist constants $p>1$, $C>0$, and $D>0$, and a sequence of $\mathcal{C}^\infty$-surfaces $\{(\Sigma_i, g_i)\}_{i\in \mathbb{N}}$, such that $\diam (\Sigma_i)\leq D$, $\bar k(p,K)_{\Sigma_i} \leq C$, and 
    \[\lim_{i\rightarrow \infty} \lambda_1\left( \Sigma_i \right) =0.\]
\end{theorem}

\begin{remark}
    Along the sequence, the volume (area) of the surfaces is also uniformly bounded
    \[0<a\leq \vol(\Sigma_i)\leq A,\]
    for some constants $A,a>0$. See Lemma \ref{diamvolbounds} and the proof of Theorem \ref{curvature-uniformly-bounded} for more details.
\end{remark}
\begin{remark}
    The condition $p>1$ is equivalent to $p>\frac{n}{2}$ in dimension $n=2$, as in the situation of Theorem \ref{mainthm}. In our proof below, we consider an example with $p=3/2$, although the proof could be modified to allow larger values of $p$.
\end{remark}
\begin{remark} \label{rmk:ricci-and-gaussian-curvature}
    In a general manifold $M^n$, if $v,e_1,\ldots , e_{n-1}$ form an orthonormal basis of $T_xM$, then
    \[\Ric(v,v) = \sum_{i=1}^{n-1} \Sec(v,e_i),\]
    where $\Sec$ denotes the Sectional curvature. $\Sec(v,e_i)$ only depends on the plane spanned by $v$ and $e_i$, not on the particular pair of vectors. Hence, when $n=2$, $\Ric(v,v) = \Sec(T_xM)$, and since in dimension 2 Sectional curvature and Gaussian curvature are the same (by Gauss's Theorema Egregium), estimating $\bar k(p,K)$ can be done by estimating $L^p$ norms of the Gaussian curvature $\kappa$ of $\Sigma_i$. This is the approach of our proof (see Theorem \ref{curvature-uniformly-bounded}). Hence, our example above shows that smallness is necessary even under the more restrictive curvature assumption of integral  sectional curvature.  
\end{remark}
The paper is structured as follows. In Section \ref{sec:eigenvalue-estimate} we review the eigenvalue estimate for sequences of surfaces of revolution with a dumbbell-like shape. We derive the estimates without a specific profile curve. In Section \ref{sec:examples-dumbbells} we construct an explicit sequence of surfaces and provide estimates for $\bar k(p,K)$ in Theorem \ref{curvature-uniformly-bounded}, which is the main technical result of the paper. We also derive estimates on their area and diameter, and conclude the proof of Theorem \ref{mainthm}. We finish with an Appendix, where we include the proofs of some technical claims used in Section~\ref{sec:examples-dumbbells}, not included in the main text to make the article more readable. 

\section{Eigenvalue Estimate} \label{sec:eigenvalue-estimate}
In the celebrated paper \cite{Cheeger}, Cheeger explains an example from Calabi, where they consider a sequence of dumbbell shaped surfaces whose first eigenvalue goes to $0$. This is portrayed as an example to show that some sort of curvature assumption is needed to derive a positive lower bound for $\lambda_1$. Our example is based on their construction, but we need to explicitly construct the dumbbells to show that $\bar k(p,K)$ is uniformly bounded. We will take care of the construction of these surfaces in the next section. 
 
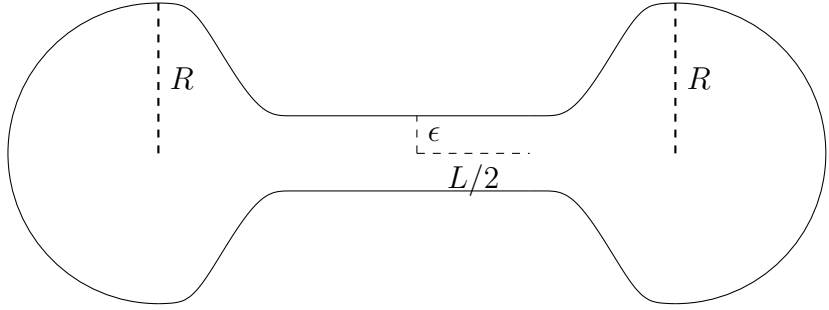
\begin{figure}[h]
    \centering

  \begin{tikzpicture}
    \draw (-\drawingXOne,-\drawingEps) -- (\drawingXOne,-\drawingEps);
    \draw (-\drawingXOne,\drawingEps) -- (\drawingXOne,\drawingEps);

    \draw[-] plot[smooth] coordinates {%
        (1.500000,0.500000)
        (1.600000,0.500000)
        (1.700000,0.500094)
        (1.800000,0.503319)
        (1.900000,0.521143)
        (2.000000,0.566093)
        (2.100000,0.642731)
        (2.200000,0.748077)
        (2.300000,0.875935)
        (2.400000,1.020057)
        (2.500000,1.175142)
        (2.600000,1.336341)
        (2.700000,1.497998)
        (2.800000,1.652165)
        (2.900000,1.787726)
        (3.000000,1.891874)
        (3.100000,1.956021)
        (3.200000,1.984680)
        (3.300000,1.995334)
        (3.400000,1.999667)
        (3.436492,1.999667)
      };
    \draw[-] plot[smooth] coordinates {%
        (1.500000,-0.500000)
        (1.600000,-0.500000)
        (1.700000,-0.500094)
        (1.800000,-0.503319)
        (1.900000,-0.521143)
        (2.000000,-0.566093)
        (2.100000,-0.642731)
        (2.200000,-0.748077)
        (2.300000,-0.875935)
        (2.400000,-1.020057)
        (2.500000,-1.175142)
        (2.600000,-1.336341)
        (2.700000,-1.497998)
        (2.800000,-1.652165)
        (2.900000,-1.787726)
        (3.000000,-1.891874)
        (3.100000,-1.956021)
        (3.200000,-1.984680)
        (3.300000,-1.995334)
        (3.400000,-1.999667)
        (3.436492,-1.999667)
      };
    \draw[-] plot[smooth] coordinates {%
        (-1.500000,0.500000)
        (-1.600000,0.500000)
        (-1.700000,0.500094)
        (-1.800000,0.503319)
        (-1.900000,0.521143)
        (-2.000000,0.566093)
        (-2.100000,0.642731)
        (-2.200000,0.748077)
        (-2.300000,0.875935)
        (-2.400000,1.020057)
        (-2.500000,1.175142)
        (-2.600000,1.336341)
        (-2.700000,1.497998)
        (-2.800000,1.652165)
        (-2.900000,1.787726)
        (-3.000000,1.891874)
        (-3.100000,1.956021)
        (-3.200000,1.984680)
        (-3.300000,1.995334)
        (-3.400000,1.999667)
        (-3.436492,1.999667)
      };
    \draw[-] plot[smooth] coordinates {%
        (-1.500000,-0.500000)
        (-1.600000,-0.500000)
        (-1.700000,-0.500094)
        (-1.800000,-0.503319)
        (-1.900000,-0.521143)
        (-2.000000,-0.566093)
        (-2.100000,-0.642731)
        (-2.200000,-0.748077)
        (-2.300000,-0.875935)
        (-2.400000,-1.020057)
        (-2.500000,-1.175142)
        (-2.600000,-1.336341)
        (-2.700000,-1.497998)
        (-2.800000,-1.652165)
        (-2.900000,-1.787726)
        (-3.000000,-1.891874)
        (-3.100000,-1.956021)
        (-3.200000,-1.984680)
        (-3.300000,-1.995334)
        (-3.400000,-1.999667)
        (-3.436492,-1.999667)
      };

    \draw (\drawingXTwoNeg,\drawingR) arc (90:270:\drawingR);
    \draw (\drawingXTwo,\drawingR) arc (90:-90:\drawingR);

    \draw[dashed,thick] (\drawingXTwoNeg,0) -- node[right] {$R$} (\drawingXTwoNeg,\drawingR);
    \draw[dashed,thick] (\drawingXTwo,0) -- node[right] {$R$} (\drawingXTwo,\drawingR);

    \draw[dashed] (0,0) -- node[right] {$\epsilon$} (0,\drawingEps);
    \draw[dashed] (0,0) -- node[below] {$L/2$} (\drawingXOne,0);
  \end{tikzpicture}
  \caption{Dumbbell surface}
    \label{fig:dumbbell}
\end{figure}

Consider a dumbbell shaped surface $\Sigma\subseteq \R^3$ (see Fig. \ref{fig:dumbbell}), with the Riemannian metric inherited from $\R^3$, formed by two halves of spheres of radius $R$, joined by two neck regions to a thin cylinder of length $L$ and radius $\epsilon>0$.

Suppose that there exists a constant $D>0$ such that $\diam(\Sigma)\leq D$, where $\diam(\Sigma)$  denotes the intrinsic diameter of $\Sigma$.

Following the proof in \cite{Cheeger}, consider a function $g$ defined to be $-c$ on the left half sphere and neck region, $+c$ on the right half sphere and neck region, and to increase linearly along the cylinder. In particular, we have
\[\int_\Sigma gdv =0, \ \int_\Sigma |\nabla g|^2 dv = 8\pi \epsilon \frac{c^2}{L} \text{ and } \int_\Sigma g^2 dv \geq c^2(4\pi R^2).\]
Then by the Rayleigh quotient characterization of $\lambda_1$, using $g$ as a test function, we have
\[\lambda_1(\Sigma) = \inf_{\int_\Sigma f = 0} \frac{\int_\Sigma |\nabla f|^2dv}{\int_\Sigma f^2 dv}\leq \frac{\int_\Sigma |\nabla g|^2dv}{\int_\Sigma g^2 dv} \leq \frac{2}{LR^2}\epsilon.\]
Thus, by considering a sequence of dumbbells $\Sigma_k$, with cylinder of radius $\epsilon_k = \frac{1}{k}$, we obtain  
\[\lim_{k\rightarrow \infty}\lambda_1(\Sigma_k) \leq \lim_{k\rightarrow \infty} \frac{2}{LR^2}\frac{1}{k}=0.\]

This is not a contradiction  with the classical Zhong-Yang estimate \eqref{ZhongYang}, since the dumbbells do not satisfy $\Ric \geq 0$, as the neck regions contain points with negative curvature. In \cite{RSWZ} they discussed why this also does not contradict the estimate \eqref{RSWZ}: the amount of negative integral curvature of the dumbbells $\bar k(p,0)$ is not small. Since by definition $\bar k(p,K)\geq \bar k(p,0)$ for $K\geq 0$, $\bar k(p,K)$  is also not small, which explains why estimate \eqref{Aubry} doesn't provide a lower bound either.

To show that the smallness of $\bar k(p,K)$ is necessary and can not be replaced by just a uniform bound, we need to show that this sequence can be constructed in such a way that $\bar k(p,K)$ is uniformly bounded for some $p>1$, which was not shown in \cite{RSWZ}, not even in the $K=0$ case. We also need the uniform bound $\diam(\Sigma_k)\leq D$ from the theorem to hold. To obtain this,  one needs a more explicit construction of the surface $\Sigma_k$, as presented in the next section. We also include a uniform bound $\vol(\Sigma_k)\leq A$ on the area of the surfaces for completion.


\section{Examples of Dumbbells}\label{sec:examples-dumbbells}
In this section we will construct the dumbbell-like surfaces as surfaces of revolution. As described in Section \ref{sec:eigenvalue-estimate}, these dumbbells have a cylindrical handle in the center of radius $\epsilon>0$ and length $L$ joined through two neck regions to a spherical cap on each end of radius $R$, assuming $\epsilon < R$. We will specify half of the surface by constructing a profile curve for $x \geq 0$, the other half being completely symmetric.

To achieve the construction, we introduce a smooth step function $s$ given by
\begin{equation}\label{s(x)}
s(x) = \begin{cases}
    0 & x \leq 0 \\
    \left( 1 + \exp{\left( \frac{1}{x} - \frac{1}{1-x}\right) } \right)^{-1} & 0 < x < 1 \\
    1 & x \geq 1.
\end{cases}    
\end{equation}
The function $s$ is smooth and satisfies the following symmetry and decay properties (see Appendix Claim~\ref{clm:stepfeats})
\[s(x) + s(1-x) = 1, \qquad \lim_{x \to 0} s^{(n)}(x) x^{-m} = 0\ \forall m,n \geq 0,\]
\[\lim_{x \to 0} \frac{s'(2x)^a s''(2x)^b}{s(x)} = 0 \text{ for any $a,b \geq 0$ s.t. $a + b > 2$.}\]

Let $x_1 = L/2$ be the $x$-coordinate of the end of the dumbbell's cylinder and define \[c_\epsilon(x) = \sqrt{R^2 - (x - \sqrt{R^2 - \epsilon^2} - x_1)^2}.\] This is an upper semicircle of radius $R$ centered at $(x_1 + \sqrt{R^2 - \epsilon^2}, 0)$, such that $c_\epsilon(x_1) = \epsilon$. For $x$ near $x_1$, we have $0 < c_\epsilon(x) - \epsilon < c_0(x)$, $0 < c_\epsilon'(x) < c_0'(x)$, and $c_0''(x) < c_\epsilon''(x) < 0$. Denoting $x_2 := x_1 + \sqrt{R^2 - \epsilon^2}$, define the smooth profile curve by
\[g_\epsilon(x) = \begin{cases}
    \epsilon & 0\leq x < x_1 \\
    \epsilon + (c_\epsilon(x) - \epsilon) s\left( \frac{x-x_1}{\sqrt{R^2 - \epsilon^2}} \right) & x_1 \leq x \leq x_2 \\
    c_\epsilon(x) & x_2 < x \leq x_2 + R.
\end{cases}\]
By construction, $g_\epsilon$ is smooth, and is monotone in $[x_1, x_2]$.

We would like to bound the $L^p$ norm of $\rho_K$ on the surface of revolution generated by $g_\epsilon$ for some $p > 1$. For convenience, we will choose $p = 3/2$, which is sufficient for our purposes. Although the proof might work for larger values of $p$, in some steps our estimates in the Appendix would need to be adjusted.
\begin{theorem} \label{curvature-uniformly-bounded}
    Consider the surface of revolution generated by $g_\epsilon$. For $\epsilon > 0$ small enough and for any $K\geq 0$, there exists $C(K,R)>0$, independent of $\epsilon$, such that $$\bar k(3/2, K)<C.$$ 
\end{theorem}
\begin{proof}
    From Remark \ref{rmk:ricci-and-gaussian-curvature}, it suffices to produce $L^p$ bounds on the Gaussian curvature $\kappa$. For a surface of revolution which has generating curve $g$, the Gaussian curvature is given by
    \[\kappa = \frac{-g''(x)}{g(x) (g'(x)^2 + 1)^2}.\]
    The negative part of the curvature, $\kappa^-$, can only be different from zero on the neck region, which corresponds to  $x_1 \leq x \leq x_2$. Notice that, using the area element $dS = 2\pi g(x) \sqrt{1+g'(x)^2}$, we have
    \begin{align}
        \norm{\kappa^-}_{3/2}^{3/2} &\leq \int_{Neck} |\kappa|^{3/2}\ dS \nonumber \\
                               &= 2\pi \int_{x_1}^{x_2} \frac{|g''(x)^{3/2}|}{g(x)^{1/2} (g'(x)^2 + 1)^{5/2}}\ dx \nonumber \\ 
                               &\leq 2\pi \int_{x_1}^{x_2} \left( \frac{|g''(x)^3|}{g(x)} \right)^{1/2}\ dx. \label{intcurvge}
    \end{align}
    The second derivative of $g_\epsilon$ on this region is given by
    \[\begin{split} g_\epsilon''(x) &=
     c_\epsilon''(x) s\left( \frac{x-x_1}{\sqrt{R^2 - \epsilon^2}} \right) + \frac{2 c_\epsilon'(x)}{\sqrt{R^2 - \epsilon^2}}  s'\left( \frac{x-x_1}{\sqrt{R^2 - \epsilon^2}} \right)\\ 
     &\quad + \frac{1}{R^2 - \epsilon^2} (c_\epsilon(x) - \epsilon) s''\left( \frac{x-x_1}{\sqrt{R^2 - \epsilon^2}} \right).    \end{split}\]
    There exist $\epsilon_M>0$ and $x_M>x_1$ such that $g_\epsilon''(x)>0$ and $s''>0$ for $(x,\epsilon)\in (x_1,x_M)\times (0,\epsilon_M)$ (see Appendix Claim~\ref{clm:dblprimepos}). Furthermore, the only negative term is $c_\epsilon''$, so dropping that term we obtain
    \begin{align*}
        0 \leq g_\epsilon''(x) \leq & \ 
        \frac{2}{\sqrt{R^2 - \epsilon^2}} c_\epsilon'(x) s'\left( \frac{x-x_1}{\sqrt{R^2 - \epsilon^2}} \right)+ \frac{1}{R^2 - \epsilon^2} (c_\epsilon(x) - \epsilon) s''\left( \frac{x-x_1}{\sqrt{R^2 - \epsilon^2}} \right) \\
        \leq & \ \frac{2}{\sqrt{R^2 - \epsilon^2}} c_0'(x) s'\left( \frac{x-x_1}{\sqrt{R^2 - \epsilon^2}} \right) + \frac{1}{R^2 - \epsilon^2} c_0(x) s''\left( \frac{x-x_1}{\sqrt{R^2 - \epsilon^2}} \right).
    \end{align*}
    Further restricting $\epsilon \leq \min \{\epsilon_M, R / 2\}$,
    we can bound this expression above by
    \[
        0 \leq g_\epsilon''(x) \leq {\Bigg (} \frac{4}{\sqrt{3} R} c_0'(x) s'\left( \frac{2(x-x_1)}{R} \right)  + \frac{4}{3 R^2} c_0(x) s''\left( \frac{2(x-x_1)}{R} \right) {\Bigg )} =: U(x),
    \]
    and this upper bound $U(x)$ is continuous and bounded on an $\epsilon$-independent interval of the form $[x_1, \xi]$ thanks to the properties of $s$. Furthermore, $g_\epsilon(x) > g_0(x)$ for $x < x_1 + \frac{R}{2}$ (see Appendix Claim~\ref{clm:gepsineq}). This leads to the upper bound
    \[\frac{g_\epsilon''(x)^3}{g_\epsilon(x)} \leq \frac{U(x)^3}{g_0(x)} = \frac{U(x)^3}{c_0(x) s\left( \frac{x-x_1}{R} \right)}\]
    which is continuous and bounded on some $\epsilon$-independent neighborhood of the form $[x_1, \xi]$, with $\xi\leq x_M$ (see Appendix Claim~\ref{clm:upperbound}).
    
    
    Then there exists $M_1(R)>0$, independent of $\epsilon$, such that
    \[\int_{x_1}^\xi \left( \frac{g_\epsilon''(x)^3}{g_\epsilon(x)} \right)^{1/2}\ dx \leq \int_{x_1}^\xi \left( \frac{U(x)^3}{c_0(x) s\left( \frac{x-x_1}{R} \right)} \right)^{1/2}\ dx \leq M_1.\]
    Since $\displaystyle \frac{(g_\epsilon'')^3}{g_\epsilon}$ is continuous in $x$ and $\epsilon$ over the compact set $[\xi, x_1 + R]\times[0, \min \{\epsilon_M, R / 2\}]$ and $x_2 \leq x_1 + R$ for all $\epsilon < R/2$, there exists $M_2(R)>0$ such that 
    \[\int_\xi^{x_2} \left( \frac{|g_{\epsilon}''(x)|^3}{g_\epsilon(x)} \right)^{1/2} dx \leq \int_\xi^{x_1+R} \left( \frac{|g_{\epsilon}''(x)|^3}{g_\epsilon(x)} \right)^{1/2} dx \leq M_2.\]
    Hence, by \eqref{intcurvge}, we have 
    \[\norm{\kappa^-}_{3/2}^{3/2} \leq 2\pi (M_1 + M_2).\]
    This is a finite bound which is independent of $\epsilon$. To obtain then an upper bound for $\bar k(3/2,0)$, we just need a uniform lower bound $a\leq \vol(\Sigma_i)$. Notice that by computing the area of the two half spheres we can obtain one such estimate, $\vol(\Sigma_i) \geq 4\pi R^2$. Thus we get 
    \[\bar k(3/2,0) \leq \left[\frac{1}{2R^2}(M_1+M_2)\right]^{2/3}=:M(R).\]
    Finally, to derive an estimate on $\bar k(3/2,K)$, notice that for any $K>0$ we have, by definition, that $0\leq \rho_K- \rho_0 \leq (n-1)K$. For dimension $n=2$, $\rho_K-\rho_0 \leq K$. Hence, by Minkovski's inequality, we have
    \small\begin{align*}
        \bar k(3/2,K) &= \left( \fint_\Sigma \rho_K^{3/2} dv \right)^{2/3} = \left( \fint_\Sigma [(\rho_K-\rho_0)+\rho_0]^{3/2} dv \right)^{2/3}\\
        &\leq \left( \fint_\Sigma (\rho_K-\rho_0)^{3/2} dv \right)^{2/3} + \left( \fint_\Sigma \rho_0^{3/2} dv \right)^{2/3}\\
        &\leq K + \bar k(3/2,0).
    \end{align*}\normalsize
    Therefore, defining $C(K,R):= K+M$, we finish the proof.

\end{proof}

Notice that, by construction, the profile curve $g$ is monotone in the neck regions. The following elementary lemma ensures then a uniform bound on the area and diameter of $\Sigma_i$. 

\begin{lemma}
\label{diameter-bound}
Suppose that $g:[a,b]\rightarrow \R$ is a positive monotone $\mathcal{C}^1$ function, and let $\Sigma$ denote the surface of revolution generated by $g$. Then 
\begin{align}
{\rm length}(g) &\leq b-a+|g(b)-g(a)|,\\
\vol(\Sigma) &\leq 2\pi \max\{g(a),g(b)\}(b-a)+\pi|g(b)^2-g(a)^2|.
\end{align}
\end{lemma}
\begin{proof}
For simplicity, we will assume that $g'(x)\geq 0$ (the proof is analogous when $g'(x)\leq 0$). In that case, using the elementary inequality $(a+b)^2 \geq a^2+b^2$ for $a,b\geq 0$, we have
\begin{align*}
    {\rm length}(g) &= \int_a^b \sqrt{1+g'(x)^2}dx \\
                    &\leq \int_a^b 1+g'(x)dx = (b-a)+g(b)-g(a),
\end{align*}
and similarly
\begin{align*}
    \vol (\Sigma) &= \int_a^b 2\pi g(x)\sqrt{1+g'(x)^2}dx \\
                  &\leq 2\pi\int_a^b  g(x)+ g(x)g'(x)dx \\
                  &\leq 2\pi g(b)(b-a) + \pi (g(b)^2-g(a)^2).
                  \end{align*}
                  \end{proof}
\begin{corollary} \label{diamvolbounds}
In the example constructed above, for $\epsilon =\frac{1}{i}\leq 1$, we have
\begin{align}
    \diam(\Sigma_i)&\leq (2\pi+8)R+2L,\\
    \vol(\Sigma_i)&\leq 10\pi R^2+2\pi L.
\end{align}
\end{corollary}
\begin{proof}
The diameter of $\Sigma_i$ is at most twice the distance between the antipodal poles of the spheres. The distance between those can be estimated by half of a great circle of the sphere of radius $R$, plus the length $L$ of the cylinder, plus two times the length of the profile curve in a neck region. Similarly, the area of $\Sigma_i$ is at most the area of a sphere of radius $R$, plus the area of the cylinder of radius at most $1$ and lenght $L$, plus the area of the two neck regions. Using \ref{diameter-bound} to bound the length and area of the neck regions, we obtain the estimates.
\end{proof}

This leads us to our main theorem.

\begin{proof}[Proof of Theorem \ref{mainthm}]
    Fix some $R > 1$ and $L>0$, and consider the sequence of surfaces of revolution $\{(\Sigma_i, g_i)\}_{i\in\mathbb{N}}$ described above, with $\epsilon = \frac{1}{i}$. Notice that Corollary \ref{diamvolbounds} and Theorem \ref{curvature-uniformly-bounded} show that there exist constants $C,D>0$, independent of $\epsilon$, such that $\bar{k}(3/2, K)_{\Sigma_i} \leq C$ and  $\diam(\Sigma_i) \leq D$. Finally, following the discussion in Section \ref{sec:eigenvalue-estimate}, we have that $\lambda_1(\Sigma_i)\rightarrow 0$.
\end{proof}
\section*{Appendix}
\setcounter{section}{1}
\setcounter{equation}{0}
\renewcommand{\thesection}{\Alph{section}}

In this section we include the proofs of some claims that were used in the proofs above. 

\begin{claim}\label{clm:stepfeats}
    The function $s(x)$ defined in \eqref{s(x)} satisfies
    \[s(x) + s(1-x) = 1, \qquad \lim_{x \to 0} s^{(n)}(x) x^{-m} = 0\ \forall m,n \geq 0,\]
\[\lim_{x \to 0} \frac{s'(2x)^a s''(2x)^b}{s(x)} = 0 \text{ for any $a,b \geq 0$ s.t. $a + b > 2$.}\]
\end{claim}
\begin{proof}
    The first property follows from the symmetry of $s(x)$ about $x=1/2$, and the second property can be derived by observing that all derivatives of $s$ are $0$ at $x = 0$ and applying L'H\^{o}pital's rule $m$ times to $\frac{s^{(n)}(x)}{x^m}$. 
    For the last property, denote $f(x):=\displaystyle \frac{s'(2x)^as''(2x)^b}{s(x)}$. By direct computation, for $x>0$, we have that 
     \[f(x)=\left(\frac{rl}{(1+l)^2}\right)^a\left(\frac{2r^2l^2}{(1+l)^3}-\frac{(\tilde{r}+r^2)l}{(1+l)^2}\right)^b(1+\tilde{l}),\]
     where $\tilde{l}=\tilde{l}(x) := e^{\frac{1}{x}-\frac{1}{1-x}}$, $l=l(x) :=\tilde{l}(2x)$, $r=r(x) := \frac{1}{4x^2}+\frac{1}{(1-2x)^2}$ and $\tilde{r}=\tilde{r}(x) := \frac{1}{4x^3}- \frac{2}{(1-2x)^3}$. Since $l(x),\tilde{l}(x)\to \infty$ as $x\to 0^+$, we have that
     \begin{equation*}
         \lim_{x\to 0^+} f(x) = \lim_{x\to 0^+} r^a\left(r^2-\tilde{r} \right)^b\frac{\tilde{l}}{l^{a+b}}=0,
     \end{equation*}
     where we used that $a+b>2$ and $$ \frac{\tilde{l}(x)}{l(x)^{a+b}} = \exp\left( \frac{2-(a+b) + O(x)}{2x(1-x)(1-2x)} \right).$$
\end{proof}
\begin{claim}\label{clm:dblprimepos}
    There exist $\epsilon_M > 0$ and $x_M > x_1$ such that $g_\epsilon''(x) > 0$ and $s''(\frac{x - x_1}{\sqrt{R^2 - \epsilon^2}}) > 0$ for $(x,\epsilon) \in (x_1, x_M) \times (0, \epsilon_M)$.
\end{claim}
\begin{proof}
    Let $x^* = \frac{x-x_1}{\sqrt{R^2 - \epsilon^2}}$. Since $s(x^*) > 0$ for $x^* > 0$, $g_\epsilon''(x)$ has the same sign as $\frac{g_\epsilon''(x)}{s(x^*)}$. Expanding the definition of $g_\epsilon''(x)$, one can see that this is bounded below by
    \[c_\epsilon''(x) + \frac{2 c_\epsilon'(x)}{\sqrt{R^2 - \epsilon^2}} \frac{s'(x^*)}{s(x^*)}.\]
    If we can show that 
    \[c_\epsilon''(x) + \frac{2 c_\epsilon'(x)}{\sqrt{R^2 - \epsilon^2}} \frac{s'(x^*)}{s(x^*)} > 0\]
    on some product set $(x, \epsilon) \in (x_1, x_M) \times (0, \epsilon_M)$, then we are done. One can manipulate the preceding inequality, eliminating $x$ in favor of $x^*$, to obtain
    \begin{equation} \label{eqn:claim2-main-inequality}
        2\frac{s'(x^*)}{s(x^*)} > \frac{1}{1-x^*} \frac{1}{1 - \frac{R^2 - \epsilon^2}{R^2} (1-x^*)^2}.
    \end{equation}
    The left hand side of  \eqref{eqn:claim2-main-inequality} is 
    \[2\frac{s'(x^*)}{s(x^*)} = 2 \frac{e^{1/x^*}}{e^{1/(1-x^*)} + e^{1/x^*}} \frac{1 - 2x^* + 2{x^*}^2}{{x^*}^2 (x^* - 1)^2}.\]
    For $x^* < 1/2$, the first factor is at least $1/2$, so 
    \[2\frac{s'(x^*)}{s(x^*)} \geq \frac{1 - 2x^* + 2{x^*}^2}{{x^*}^2 (x^* - 1)^2} =: A(x^*).\]
    On the other hand, the right-hand side of \eqref{eqn:claim2-main-inequality} satisfies
    \[\frac{1}{1-x^*} \frac{1}{1 - \frac{R^2 - \epsilon^2}{R^2} (1-x^*)^2} \leq \frac{1}{(1-x^*)(2x^*-{x^*}^2)}=:B(x^*)\]
    We see that that $A(1/2) > B(1/2)$. Furthermore, by elementary means, working with the corresponding polynomial, we can see that the equation $A(x^*)=B(x^*)$
    has no solutions for $x^* \in (0, 1/2)$. This shows $A(x^*) > B(x^*)$, which implies \eqref{eqn:claim2-main-inequality}, as long as $x^* < 1/2$. 
    Restrict $\epsilon < R/2 =: \epsilon_M$. Then
    \[x^* = \frac{x-x_1}{\sqrt{R^2 - \epsilon^2}} < \frac{x-x_1}{\sqrt{R^2 - \left( \frac{R}{2} \right)^2}} < \frac{x-x_1}{R/2}.\]
    If we impose $\frac{x - x_1}{R/2} < \frac{1}{2}$, then the condition we required on $x^*$ is satisfied and we obtain the sufficient bound $x_1 < x < x_1 + \frac{R}{4} =: x_M$ which is precisely what we needed. The claim about $s''$ follows as long as $x^* < 1/2$, so it is satisfied by our choice of $x_M$. 
\end{proof}

\begin{claim}\label{clm:gepsineq}
    For $x \in (x_1,x_2)$, $g_\epsilon(x) \geq g_0(x)$.
\end{claim}
\begin{proof}
    Notice that $s\left( \frac{x-x_1}{\sqrt{R^2 - \epsilon^2}} \right)\geq s\left( \frac{x-x_1}{R} \right)$, so 
    \[g_\epsilon (x) \geq \epsilon + (c_\epsilon(x) - \epsilon) s\left( \frac{x-x_1}{R} \right)=:h_\epsilon(x).\]
    The $\epsilon$-derivative of this lower bound is
    \[\frac{\partial h_\epsilon}{\partial\epsilon}=1 + \left( \frac{\partial c_\epsilon}{\partial\epsilon} - 1 \right) s\left( \frac{x-x_1}{R} \right) \geq 0,\]
    since $\frac{\partial c_\epsilon}{\partial \epsilon} \geq 0$ by direct computation, and $s\leq 1$.
    Thus $h_\epsilon (x) \geq h_0(x)$. Since $h_0(x) = g_0(x)$, then $g_\epsilon(x) \geq g_0(x)$.
\end{proof}

\begin{claim}\label{clm:upperbound}
    The $\epsilon$-independent quantity
    \[\frac{U(x)^3}{c_0(x) s(\frac{x - x_1}{R})}\]
    is continuous and bounded on an $\epsilon$-independent neighborhood of the form $[x_1, \xi]$ where $\xi \leq x_M$.
\end{claim}
\begin{proof}
    The stated quantity is continuous away from $x_1$, so it suffices to show that the quantity converges to $0$ as $x \to x_1^+$. The function $U(x)$ has the form $p+q$, where 
    \begin{align*}
        p &= \frac{4}{\sqrt{3} R} c_0'(x) s'\left( \frac{2(x-x_1)}{R} \right), \\
        q &= \frac{4}{3 R^2} c_0(x) s''\left( \frac{2(x-x_1)}{R} \right).
    \end{align*}
    Applying the binomial theorem, 
    \[\frac{U(x)^3}{c_0(x) s(\frac{x - x_1}{R})} = \sum_{k=0}^3 \frac{\binom{3}{k} p^k q^{3-k}}{c_0(x) s(\frac{x - x_1}{R})}.\]
    We claim that each term in the sum tends to $0$ as $x \to x_1^+$. Each term in the sum contains a ratio of $c_0$ and its derivative, and near $x_1$ this ratio is either finite or a pole of finite order. Note that the sum of powers of $s'$ and $s''$ in the numerator of each term is $3$. Thanks to this, one may factor each term into two parts in the following way. One factor comprises the ratio of $c_0$ and its derivatives, times either $(s')^{1/2}$ or $(s'')^{1/2}$; this factor goes to $0$ as $x \to x_1^+$ because derivatives of $s$ annihilate finite-order poles. The remaining factor is a ratio  whose numerator is $s'(\frac{2(x - x_1)}{R})^a s''(\frac{2(x - x_1)}{R})^b$, where $a+b = \frac{5}{2}$, and whose denominator is $s(\frac{x - x_1}{R})$; this factor also goes to $0$ as $x \to x_1^+$ because of the stated properties of $s$ (see also Appendix Claim~\ref{clm:stepfeats}). Overall this shows that 
    \[\lim_{x \to x_1^+} \frac{U(x)^3}{c_0(x) s(\frac{x - x_1}{R})} = 0\]
    and this completes the proof.
\end{proof}


\newpage 

{\footnotesize
\begin{thebibliography}{00}
\bibitem{AndrewsClutterbuck}
B.~Andrews, J.~Clutterbuck, Sharp modulus of continuity for parabolic equations
  on manifolds and lower bounds for the first eigenvalue, Anal. PDE 6~(5)
  (2013) 1013--1024.
\newblock \href {https://doi.org/10.2140/apde.2013.6.1013}
  {\path{doi:10.2140/apde.2013.6.1013}}.

\bibitem{Aubry}
E.~Aubry, Finiteness of {$\pi_1$} and geometric inequalities in almost positive
  {R}icci curvature, Ann. Sci. \'{E}cole Norm. Sup. (4) 40~(4) (2007) 675--695.
\newblock \href {https://doi.org/10.1016/j.ansens.2007.07.001}
  {\path{doi:10.1016/j.ansens.2007.07.001}}.


\bibitem{BuragoBuragoIvanov}
D.~Burago, Y.~Burago, S.~Ivanov, A course in metric geometry, Graduate Studies in Mathematics, American Mathematical Society (2001), ISBN 0-8218-2129-6, 
\newblock \href {https://doi.org/10.1090/gsm/033}
  {\path{10.1090/gsm/033}}.


\bibitem{Cheeger}
J.~Cheeger, A lower bound for the smallest eigenvalue of the {L}aplacian, in:
  Problems in analysis ({P}apers dedicated to {S}alomon {B}ochner, 1969), 1970,
  pp. 195--199.
  
  \bibitem{CheegerColding}
  J.~Cheeger, T.~Colding, On the structure of spaces with {R}icci curvature bounded below. {III}, J. Differential Geom., no. 1, vol.54, pp. 37--74 (2000),  \newblock \href{http://projecteuclid.org/euclid.jdg/1214342146}{http://projecteuclid.org/euclid.jdg/1214342146}
  
  \bibitem{ChenWang}
  M.~Chen, F.~Wang, General formula for lower bound of the first eigenvalue on Riemannian manifolds. Sci. China Ser. A 40(4), 384-–394 (1997)
  \newblock \href{https://doi.org/10.1007/BF02911438}{\path{doi:10.1007/BF02911438}}

\bibitem{DaiWeiZhang}
X.~Dai, G.~Wei, Z.~Zhang, Local {S}obolev constant estimate for integral
  {R}icci curvature bounds, Adv. Math. 325 (2018) 1--33.
\newblock \href {https://doi.org/10.1016/j.aim.2017.11.024}
  {\path{doi:10.1016/j.aim.2017.11.024}}.

\bibitem{Gallot}
S.~Gallot, Isoperimetric inequalities based on integral norms of {R}icci
  curvature, no. 157-158, 1988, pp. 191--216, colloque Paul L\'{e}vy sur les
  Processus Stochastiques (Palaiseau, 1987).

\bibitem{HangWang}
F.~Hang, X.~Wang, A remark on {Z}hong-{Y}ang's eigenvalue estimate, Int. Math. Res. Not. IMRN, no. 18 (2017). 
\newblock \href {https://doi.org/10.1093/imrn/rnm064}
  {\path{doi:10.1093/imrn/rnm064}}.

\bibitem{Kroger}
P.~Kr\"oger, On the spectral gap for compact manifolds, J. Differ. Geom. 36(2), 315–-330 (1992)
\newblock \href{http://projecteuclid.org/euclid.jdg/1214448744}{\path{http://projecteuclid.org/euclid.jdg/1214448744}}

\bibitem{Leal}
H.~Leal, Spectral gaps on {R}iemannian {M}anifolds, ProQuest LLC, Ann Arbor,
  MI, 2020, thesis (Ph.D.)--University of California, Irvine.

\bibitem{LiYau}
P.~Li, S.~T. Yau, Estimates of eigenvalues of a compact {R}iemannian manifold,
  in: Geometry of the {L}aplace operator ({P}roc. {S}ympos. {P}ure {M}ath.,
  {U}niv. {H}awaii, {H}onolulu, {H}awaii, 1979), Proc. Sympos. Pure Math.,
  XXXVI, Amer. Math. Soc., Providence, R.I., 1980, pp. 205--239.

\bibitem{Lichnerowicz}
A.~Lichnerowicz, G\'{e}om\'{e}trie des groupes de transformations, Travaux et
  Recherches Math\'{e}matiques, III. Dunod, Paris, 1958.

\bibitem{LingLu}
J.~Ling, Z.~Lu, Bounds of eigenvalues on {R}iemannian manifolds, in: Trends in
  partial differential equations, Vol.~10 of Adv. Lect. Math. (ALM), Int.
  Press, Somerville, MA, 2010, pp. 241--264.

\bibitem{Petersen-Sprouse1998}
P.~Petersen, C.~Sprouse, Integral curvature bounds, distance estimates and
  applications, J. Differential Geom. 50~(2) (1998) 269--298.

\bibitem{PetersenWei}
P.~Petersen, G.~Wei, Relative volume comparison with integral curvature bounds,
  Geom. Funct. Anal. 7~(6) (1997) 1031--1045.
\newblock \href {https://doi.org/10.1007/s000390050036}
  {\path{doi:10.1007/s000390050036}}.
  
\bibitem{Obata}
M.~Obata, Certain conditions for a {R}iemannian manifold to be isometric with a sphere, J. Math. Soc. Japan, vol. 14 (1962), pp. 333--340. 
\newblock \href {https://doi.org/10.2969/jmsj/01430333}
  {\path{doi:10.2969/jmsj/01430333}}.

\bibitem{Ramosthesis}
X.~Ramos~Oliv\'{e},
Gradient {E}stimates {U}nder {I}ntegral {C}urvature {C}onditions, Thesis (Ph.D.)--University of California, Riverside, ProQuest LLC, Ann Arbor, MI, (2019).
\newblock \href{http://gateway.proquest.com/openurl?url_ver=Z39.88-2004&rft_val_fmt=info:ofi/fmt:kev:mtx:dissertation&res_dat=xri:pqm&rft_dat=xri:pqdiss:13895585}{Access here}.

\bibitem{RSWZ}
X.~Ramos~Oliv\'{e}, S.~Seto, G.~Wei, Q.~S. Zhang, Zhong-{Y}ang type eigenvalue
  estimate with integral curvature condition, Math. Z. 296~(1-2) (2020)
  595--613.
\newblock \href {https://doi.org/10.1007/s00209-019-02448-w}
  {\path{doi:10.1007/s00209-019-02448-w}}.

\bibitem{RoseWei}
C.~Rose, G.~Wei, Eigenvalue estimates for {Kato}-type {Ricci} curvature
  conditions (2020).
\newblock \href {http://arxiv.org/abs/2003.07075} {\path{arXiv:2003.07075}}.

\bibitem{ShiZhang}
Y.M.~Shi, H.C.~Zhang, Lower bounds for the first eigenvalue on compact manifolds, Chin. Ann. Math. Ser. A 28(6), 863-–866 (2007). (Chinese, with English and Chinese summaries)

\bibitem{SormaniWenger}
C.~Sormani, S.~Wenger, The intrinsic flat distance between {R}iemannian manifolds and other integral current spaces, J. Differential Geom., vol. 87 (2011), no. 1, pp. 117--199
\newblock \href {http://projecteuclid.org/euclid.jdg/1303219774} {http://projecteuclid.org/euclid.jdg/1303219774}.

\bibitem{Wu}
J.-Y. Wu, Comparison geometry for integral {B}akry-\'{E}mery {R}icci tensor
  bounds, J. Geom. Anal. 29~(1) (2019) 828--867.
\newblock \href {https://doi.org/10.1007/s12220-018-0020-8}
  {\path{doi:10.1007/s12220-018-0020-8}}.

\bibitem{Yang}
H.~C. Yang, Estimates of the first eigenvalue for a compact {R}iemann manifold,
  Sci. China Ser. A 33~(1) (1990) 39--51.

\bibitem{ZhangWang}
Y.~Zhang, K.~Wang, An alternative proof of lower bounds for the first
  eigenvalue on manifolds, Math. Nachr. 290~(16) (2017) 2708--2713.
\newblock \href {https://doi.org/10.1002/mana.201600388}
  {\path{doi:10.1002/mana.201600388}}.

\bibitem{ZhongYang}
J.~Q. Zhong, H.~C. Yang, On the estimate of the first eigenvalue of a compact
  {R}iemannian manifold, Sci. Sinica Ser. A 27~(12) (1984) 1265--1273.

\end{thebibliography}}

{\footnotesize  
\medskip
\medskip
\vspace*{1mm} 
 
\noindent {\it Connor C. Anderson}\\  
Department of Mathematical Sciences, Worcester Polytechnic Institute\\
100 Institute Road\\
Worcester, MA 01609\\
E-mail: {\tt ccanderson@wpi.edu}\\ \\  

\noindent {\it Xavier Ramos Oliv\'e, PhD}\\  
Department of Mathematical Sciences, Worcester Polytechnic Institute\\
100 Institute Road\\
Worcester, MA 01609\\
E-mail: {\tt xramosolive@wpi.edu}\\ \\

\noindent {\it Kamryn Spinelli}\\  
Department of Mathematical Sciences, Worcester Polytechnic Institute\\
100 Institute Road\\
Worcester, MA 01609\\
E-mail: {\tt kpspinelli@wpi.edu} \\ \\   

}


\end{document}